\documentclass[12pt]{article}
\usepackage{amsmath, amsthm, amssymb,enumitem}
\usepackage[top=1cm,left=1.5cm,right=1cm,bottom=1.5cm]{geometry}

\theoremstyle{plain}
\newtheorem{theorem}{Theorem}
\newtheorem{lemma}[theorem]{Lemma}

\newtheorem{corollary}[theorem]{Corollary}

\usepackage{array,color,colortbl}
\providecommand{\mk}{\cellcolor[gray]{.8}}
\providecommand{\MK}{\cellcolor[gray]{.5}}

\def\eref#1{$(\ref{#1})$}
\def\sref#1{\S$\ref{#1}$}
\def\lref#1{Lemma~$\ref{#1}$}
\def\tref#1{Theorem~$\ref{#1}$}

\def\cref#1{Conjecture~$\ref{#1}$}
\def\Cref#1{Corollary~$\ref{#1}$}

\renewcommand{\geq}{\geqslant}
\renewcommand{\leq}{\leqslant}
\renewcommand{\ge}{\geqslant}
\renewcommand{\le}{\leqslant}
\renewcommand{\emptyset}{\varnothing}
\def\Z{\mathbb{Z}}
\def\O{\mathcal{O}}
\def\F{\mathcal{F}}
\def\T{\mathcal{T}}
\def\={\equiv}

\def\dfrac#1#2{\lower0.15ex\hbox{\large$\frac{#1}{#2}$}}

\title{Mutually orthogonal binary frequency squares}

\author{
Thomas Britz \thanks{
School of Mathematics and Statistics, UNSW Sydney, NSW 2052, Australia.}\\
\texttt{britz@unsw.edu.au}\\
\and
Nicholas J. Cavenagh\thanks{
Department of Mathematics,
 The University of Waikato,
 Private Bag 3105,
 Hamilton 3240, New Zealand.}\\
\texttt{nickc@waikato.ac.nz}\\ 
\and
Adam Mammoliti
\thanks{
School of Mathematics, Monash University, Clayton 3800, Australia.}\\ 
\texttt{adam.mammoliti@monash.edu}\\ 
\and
Ian M. Wanless
\footnotemark[3]\\ 
\texttt{ian.wanless@monash.edu}\\
}

\begin{document}

\date{}
\maketitle

\begin{abstract}
A \emph{frequency square} is a matrix in which each row and column is a 
permutation of the same multiset of symbols.
We consider only {\em binary} frequency squares of order $n$ with
$n/2$ zeroes and $n/2$ ones in each row and column.  Two such
frequency squares are \emph{orthogonal} if, when superimposed, each of
the 4 possible ordered pairs of entries occurs equally often. In this
context we say that a $k$-MOFS$(n)$ is a set of $k$ binary
frequency squares of order $n$ in which each pair of squares is
orthogonal.

A $k$-MOFS$(n)$ must satisfy $k\le(n-1)^2$, and any MOFS achieving
this bound are said to be \emph{complete}.  For any $n$ for which
there exists a Hadamard matrix of order $n$ we show that there exists
at least $2^{n^2/4-O(n\log n)}$ isomorphism classes of complete
MOFS$(n)$.  For $2<n\equiv2\pmod4$ we show that there exists a
$17$-MOFS$(n)$ but no complete MOFS$(n)$.

A $k$-maxMOFS$(n)$ is a $k$-MOFS$(n)$ that is not contained in any 
$(k+1)$-MOFS$(n)$. By computer enumeration, we establish that there
exists a $k$-maxMOFS$(6)$ if and only if $k\in\{1,17\}$ or $5\le k\le
15$.  We show that up to isomorphism there is a unique
$1$-maxMOFS$(n)$ if $n\equiv2\pmod4$, whereas no $1$-maxMOFS$(n)$
exists for $n\equiv0\pmod4$.  We also prove that there exists a
$5$-maxMOFS$(n)$ for each order $n\equiv 2\pmod{4}$ where $n\geq 6$.
\end{abstract}

\noindent {\bf MSC 2010 Codes: 05B15}

\noindent {Keywords: Frequency square; MOFS; Hadamard matrix; relation; trade.}  

\section{Introduction}

In what follows, rows and columns of an $n\times n$ array are each
indexed by $N(n)=\{1,2,\dots,n\}$.
A \emph{frequency square} $L$ of \emph{type}
$F(n;\lambda_1,\lambda_2,\dots,\lambda_m)$ is an $n\times n$ array
such that symbol $i$ occurs $\lambda_i$ times in each row and
$\lambda_i$ times in each column for each $i\in N(m)$; necessarily
$\sum_{i=1}^m\lambda_i=n$.  In the case where
$\lambda_1=\lambda_2=\dots=\lambda_m=\lambda$ we say that $L$ is of
type $F(n;\lambda)$.  A
frequency square of type $F(n;1)$ is a \emph{Latin square} of
order~$n$.  Two frequency squares of type
$F(n;\lambda_1,\lambda_2,\dots,\lambda_m)$ are \emph{orthogonal} if
each ordered pair $(i,j)$ occurs $\lambda_i\lambda_j$ times when the
squares are superimposed.  A set of \emph{mutually orthogonal
  frequency squares} (MOFS) is a set of frequency squares in which
each pair of squares is orthogonal.

Research into frequency squares focuses mainly on constructing sets of
MOFS, motivated originally by problems in statistical experiment
design.  Hedayat, Raghavarao and Seiden~\cite{HRS} showed that the
maximum possible size of a set of MOFS of type $F(n;\lambda)$ is
$(n-1)^2/(m-1)$; such a set is called \emph{complete}.  Complete MOFS
of type $F(n;\lambda)$ for $n/\lambda>2$ are only known to exist when
$n$ is a prime power~\cite{LM,LZD,Mav,street79}; a unified theory for
all known constructions is given in~\cite{JMM}.

Starting with a set of mutually orthogonal Latin squares and replacing
some subset of the symbols by zeroes, and replacing all other symbols
by ones, we can obtain a set of binary MOFS.  A slightly less obvious
connection between binary MOFS and other designs is the following.  An
{\em equidistant permutation array} $A(n,d;k)$ is a $k\times n$ array
in which each row contains each integer from $1$ to $n$ exactly once
and any two distinct rows differ in exactly $d$ positions.  One can
construct $k$ MOFS
of type $F(n;n-1,1)$ by writing down the permutation matrices that
correspond to the rows of an $A(n,n-1;k)$.
It is known from \cite{DV} that there exists an $A(n,n-1;2n-4)$ for
any $n\geq 6$ and from \cite{CRC} that there exists an
$A(q^2+q+1,q^2+q,q^3+q^2)$ for any prime power $q$.

Two sets of MOFS are \emph{isomorphic} if one can be obtained from the
other by some sequence of the following operations:
\begin{itemize}[topsep=4pt,partopsep=0pt,itemsep=2pt,parsep=2pt]
  \item Applying the same permutation to the rows of all squares in the set.
  \item Applying the same permutation to the columns of all squares in the set.
  \item Transposing all squares in the set.
  \item Permuting the symbols in one of the squares.
  \item Permuting the squares within the set (in cases where we have imposed an
    order on the set).
\end{itemize}
Isomorphism is an equivalence relation and the equivalence classes it
induces are \emph{isomorphism classes}.

For the remainder of this paper we restrict ourselves to frequency
squares of type $F(n;\lambda)$ where $n/\lambda=2$. In other words,
our squares have just two symbols, which we will take to be 0 and 1.
We will not say it each time, but all subsequent mention of MOFS will
refer to these binary MOFS.  As we are assuming that both symbols
must occur equally often within each row, the order of our MOFS must
be even.  We use MOFS$(n)$ to denote MOFS of order $n$. If there are
$k$ MOFS in the set then we write $k$-MOFS$(n)$.

The following result was
proved by Federer~\cite{Fe}; see also~\cite{street79}.

\begin{theorem}\label{t:Hadamard}
  If there exists a Hadamard matrix of order $n$,
  then there exists a complete {\rm MOFS}$(n)$.
\end{theorem}

The Hadamard conjecture famously asserts the existence of Hadamard
matrices for all orders that are divisible by~4.  If true, this would
imply the existence of a complete MOFS$(n)$ for every $n$
divisible by $4$.
Conversely, Theorem~4.6 in~\cite{deL} gives the asymptotic result that
if $n$ is divisible by $4$, then there exists a $k$-MOFS$(n)$ where
$k=n^2(1-o(1))/9$, providing a quadratic lower bound on the size of
the largest set of MOFS$(n)$.  No corresponding bound is known when
$n/2$ is odd. Indeed, very little seems to be known about MOFS$(n)$
when $n\equiv 2\pmod4$. Our primary aim in this paper is to shed some
light on this case. For example, we will show in \sref{s:rel} that
there are no complete MOFS of this type. The corresponding
problem for Latin squares is a famous problem that remains open; it
would imply the non-existence of a finite projective plane of order
$n\equiv 2\pmod4$ with $n>2$.

Maintaining consistency with Latin square
terminology, a \emph{bachelor} frequency square $F$ is one such that
there exists no frequency square $F'$ orthogonal to $F$. 
In general, a set $\{F_1,F_2,\dots,F_k\}$ of $k$-MOFS$(n)$ is said to be
\emph{maximal} if there does not exist a frequency square $F$ that is
orthogonal to $F_i$ for each $1\leq i\leq k$. 
 If we wish to specify that a $k$-MOFS$(n)$ is maximal
we may write $k$-maxMOFS$(n)$.

The structure of the paper is as follows.  In \sref{s:rel} we
demonstrate a condition that is sufficient to show that a set of MOFS
is maximal.  The condition is called a \emph{relation} and is modelled
on similar work that has been done for Latin squares. In \sref{s:bach}
we show that bachelor frequency squares are unique up to isomorphism
for orders that are $2\pmod4$ and do not exist for orders that are
$0\pmod4$. The bachelor frequency squares are maximal because they
satisfy a relation. The contrast with Latin squares is worth noting. It is
known from \cite{Ev,WW06} that bachelor Latin squares exist for all
orders $n>3$.  Moreover, there are vast numbers of bachelor Latin
squares up to isomorphism \cite{CW17}.  In \sref{s:trades} we study
small local changes that can convert a set of MOFS into a
non-isomorphic set of MOFS.  Using these ``trades'' we show that for
any $n$ for which there exists a Hadamard matrix of order $n$ there
are at least $2^{n^2/4-O(n\log n)}$ isomorphism classes of complete
MOFS$(n)$. This contrasts nicely with the result in
\sref{s:rel} that there are no complete MOFS$(n)$ when
$n\equiv2\pmod4$ and $n>2$.  In \sref{s:comp} we report on computer
enumerations for MOFS of small order. We find that aside from the
unique bachelor there are no $k$-maxMOFS$(6)$ with $k<5$.  Also, most
but not all of the $5$-maxMOFS$(6)$ satisfy a relation.  The largest
$k$-MOFS$(6)$ have size $k=17$, and they also satisfy relations. In
\sref{s:embed} we show how $k$-MOFS$(n)$ can sometimes by embedded in
$k$-MOFS$(n')$ for some $n'>n$. Using this technique we show that
there exist $17$-MOFS$(n)$ for all $n\equiv2\pmod4$ such that
$n>2$. Then in \sref{s:max} we use similar ideas to show that there
exist $5$-maxMOFS$(n)$ for all $n\equiv2\pmod4$ such that
$n>2$. Finally, in \sref{s:conc} we discuss some interesting questions
that have been prompted by our work.

\section{Relations}\label{s:rel}

The technique of relations developed in this section is based on the
techniques used in \cite{Do} and \cite{DH} (with origins in \cite{St})
to analyse the maximal sets of mutually orthogonal Latin squares.

A set $\{F_1,\dots,F_k\}$ of $k$-MOFS of order $n$ can be written as an
$n^2\times(k+2)$ orthogonal array $\O$ in which there is a row
\begin{equation}\label{e:rowofOA}
\big[i,j,F_1[i,j],F_2[i,j],\dots,F_k[i,j]\big]\,,
\end{equation}
for each $i \in N(n)$ and $j \in N(n)$. In this context it is safest
to consider MOFS to have an indexing that implies an ordering
on the squares (and hence the order of the columns in $\O$ is
well-defined.  Let $Y_c$ be the set of symbols that occur in column
$c$ of $\O$.  Then a \emph{relation} is a $(k+2)$-tuple
$(X_1,\dots,X_{k+2})$ of sets such that $X_i\subseteq Y_i$ for $1\le
i\le{k+2}$, with the property that every row~\eref{e:rowofOA} of $\O$
has an even number of columns $c$ for which the symbol in column $c$
is an element of $X_c$.  A relation is \emph{trivial on column} $c$ if
$X_c=\emptyset$ or $X_c=Y_c$. We will say that a relation is
\emph{non-trivial} if there is some column on which it is not trivial,
and \emph{full} if it is non-trivial on every column except possibly
one of the first two. We say that a relation is an $(a,b)$-relation if
$|X_1|=a$ and $|X_2|=b$. In the results below we will show certain
restrictions on the kinds of relations that can be achieved.

If we start with a relation and two of the $X_c$'s are replaced by
their complements, then we obtain another relation. In our context,
$X_c\subseteq\{0,1\}$ for $c\ge3$.  By complementing $X_c$ and $X_1$
if necessary, we may assume that $X_c=\{1\}$ or $X_c=\emptyset$ for
all $c \geq 3$. In the latter case, we have a relation on a proper
subset of the MOFS. The choices of $X_1$ and $X_2$ govern which
rows and columns of the MOFS are involved in the relation.  However,
we are only interested in properties of MOFS up to isomorphism.  That
means that only 3 quantities really matter to us for a relation: How
many MOFS are involved, how many rows are involved and how many
columns are involved.  These observations will allow us to rule out
existence of relations in a number of cases below. They also allow us
to provide an easy characterisation of MOFS that satisfy a relation.

\begin{lemma}\label{l:charactrel}
A set of {\rm MOFS} satisfies a non-trivial relation if and only if some
non-empty subset of the {\rm MOFS} have a $\Z_2$-sum that, up to
permutation of the rows and columns, has the following structure of
constant blocks:
\begin{equation}\label{e:blocks}
\left[
\begin{array}{cc}
\mathbf{0}&\mathbf{1}\\
\mathbf{1}&\mathbf{0}
\end{array}
\right].
\end{equation}
\end{lemma}

\begin{proof} 
Let the set of MOFS $\{F_1,\ldots, F_k\}$ satisfy a full relation
$(X_1,\ldots, X_{k+2})$.  For $r,c \in N(n)$, let $x_{rc}$ be the sum
over the entries in the cell $(r,c)$ of each of the squares.  Then, by
the definition of a relation, $x_{rc}\=0\pmod 2$ if $r\in X_1$ and
$c\in X_2$, or if $r\notin X_1$ and $c\notin X_2$, and
$x_{rc}\=1\pmod{2}$ if exactly one of $r\in X_1$ and $c\in X_2$ holds.
By permuting the rows and columns appropriately, the
$\mathbb{Z}_2$-sum of $F_1,\ldots, F_k$ has a structure of the form in~\eref{e:blocks}.

If the set of MOFS $\{ F_1,\ldots, F_k\}$ has a $\mathbb{Z}_2$-sum
that is, up to permutation of the rows and columns, of the form
\eref{e:blocks}, then let $X_1$ be the set of rows and $X_2$ be the
set of columns corresponding to the upper left $\mathbf{0}$ block
in~\eref{e:blocks}.  Then it is easy to check that $(X_1,X_2,\ldots,
X_{k+2})$ with $ X_{c} = \{ 1\}$ for $c \geq 3$ is a full relation on
$\{ F_1,\ldots, F_k\}$.

Thus, we have shown that a set of MOFS $\{F_1,\ldots, F_k\}$ satisfy
a full relation if and only if the $\mathbb{Z}_2$-sum of $F_1,\ldots,F_k$ 
is of the form \eref{e:blocks} up to permutation of the rows and columns.
As a set of MOFS satisfies a non-trivial relation if and only if 
a non-empty subset satisfies a full relation, the lemma follows.
\end{proof}

We stress that the blocks in~\eref{e:blocks} are allowed to be degenerate. 
For example, here are 2-MOFS(4) that satisfy a relation with 
$X_1=\{1,2\}$ and $X_2=\emptyset$. Their $\Z_2$-sum is also given.
\[
\left[
\begin{array}{cccc}
0&1&0&1\\
1&0&1&0\\
0&1&0&1\\
1&0&1&0
\end{array}
\right]
+
\left[
\begin{array}{cccc}
1&0&1&0\\
0&1&0&1\\
0&1&0&1\\
1&0&1&0
\end{array}
\right]
\equiv
\left[
\begin{array}{cccc}
1&1&1&1\\
1&1&1&1\\
0&0&0&0\\
0&0&0&0
\end{array}
\right].
\]

\begin{lemma}\label{l:parityrel}
Let $\F$ be a $k$-{\rm MOFS}$(2\lambda)$ and let
$R=(X_1,\dots,X_{k+2})$ be a full relation on $\F$.
Then $|X_1|\=|X_2|\=\lambda k\pmod{2}$.
\end{lemma}

\begin{proof}
Let $\O$ be the orthogonal array corresponding to $\F$.
Let 
$$s = \big|\{(r,c)\in N(4\lambda^2)\times N(k+2):\, \O[r,1]=1, \O[r,c] \in X_c\}\big|$$ 
be the number of cells containing symbols in the relation 
among the rows of $\O$ for which $i=1$ in \eref{e:rowofOA}.
By noting that $X_c = \{1 \}$ for $c\geq 3$, one finds that 
$s=|X_2|+\lambda k+2\lambda\delta$
where $\delta=1$ if $1\in X_1$ and $\delta=0$ otherwise.
By the definition of relations, it follows that 
$|X_2|+\lambda k$ is even. A similar argument on the first column of the MOFS
shows that $|X_1|+\lambda k$ is even. 
\end{proof}

One reason to be interested in relations is that they can be used
to diagnose maximality of a set of MOFS.

\begin{theorem}\label{t:relmax}
Suppose $k$ and $\lambda$ are both odd. Let $\F$ be a 
$k$-{\rm MOFS}$(2\lambda)$ that satisfies a full relation. Then $\F$ is
maximal.
\end{theorem}

\begin{proof}
Let $R =(X_1,\ldots, X_{k+2})$ be a full relation of $\F$. 
By \lref{l:parityrel}, we know that $|X_1|$ and $|X_2|$ are both odd.
Suppose that $\F$ can be extended by appending a square $F$. Let $F$ have
$x$ ones in cells $(r,c)$ where $r\in X_1$ and $c\in X_2$. Then $F$
has $\lambda|X_1|-x$ ones in cells $(r,c)$ where $r\in X_1$ and
$c\notin X_2$ and $\lambda|X_2|-x$ ones in cells $(r,c)$ where
$r\notin X_1$ and $c\in X_2$. 
By orthogonality, the number of pairs $(F[r,c],F_i[r,c])$ such that
$(F[r,c],F_i[r,c])=(1,1)$ is $k \lambda^2$.  On the other hand, for a
fixed pair $(r,c)$ with $F[r,c] =1$ the number of $i$'s such that
$(F[r,c],F_i[r,c])=1$ is even if either $r\in X_1$ and $c\in X_2$ or
$r\not\in X_1$ and $c\not\in X_2$, and is odd if exactly one of $r\in
X_1$ and $c\in X_2$ holds.
Therefore, we must have that
\[
1\=k\lambda^2\=\lambda|X_1|-x+\lambda|X_2|-x\=0\pmod{2}.
\]
This contradiction proves the theorem.
\end{proof}

With~\tref{t:relmax} as motivation, we now consider what relations
are possible.  Our first result is a non-existence result.

\begin{theorem}\label{t:no2or3mod4}
  Suppose that $\lambda$ is odd and $k\not\=1\pmod{4}$. 
  Then no $k$-{\rm MOFS}$(2\lambda)$ satisfies a full relation.
\end{theorem}

\begin{proof}
Let $\{F_1,\ldots, F_k \}$ be a $k$-MOFS$(2 \lambda)$ that
satisfies a full relation $(X_1,\ldots, X_{k+2})$.  Let $x_i$ be the
number of $k$-tuples in the superposition of the $k$-MOFS$(2\lambda)$
that contain exactly $i$ ones. Since each square contains $2\lambda^2$
ones, we know that
\begin{equation}\label{e:ixi}
    2k\lambda^2=\sum_{i=0}^k ix_i.
\end{equation}
Also each pair of squares has $\lambda^2$ cells where both contain a one, so
\begin{equation}\label{e:ic2xi}
    \binom{k}{2}\lambda^2=\sum_{i=0}^k \binom{i}{2}x_i.
\end{equation}
Adding twice~\eref{e:ic2xi} to~\eref{e:ixi} we find that
\begin{equation}\label{e:oddxi}
  k(k+1)\=k(k+1)\lambda^2=  \sum_{i=0}^k i^2x_i \=  \sum_{\text{odd }i}x_i\pmod{4}.
\end{equation}
However, by Lemma~\ref{l:parityrel}, the existence of the relation enforces
\begin{equation}\label{e:oddxialt}
  \sum_{\text{odd }i}x_i=|X_1|(2\lambda-|X_2|)+|X_2|(2\lambda-|X_1|)
  =2\lambda(|X_1|+|X_2|)-2|X_1||X_2|\=2k^2\pmod{4}.
\end{equation}
  Equations (\ref{e:oddxi}) and
(\ref{e:oddxialt}) contradict each other when $k\=2$ or $3\pmod{4}$.
  
  \bigskip

Next we rule out the case when $k$ is even (thereby finding a second
proof for the case $k\=2\pmod{4}$). So assume that $k$ is even, and,
hence, $|X_1|$ and $|X_2|$ are both even, by \lref{l:parityrel}.

Let $\Omega_0$ (respectively $\Omega_1$) be the set of cells $(r,c)$
for which $F_1[r,c]=1$ and in which the superposition of $F_1,\ldots,F_{k}$
has an even (respectively, odd) number of ones. We claim that
$|\Omega_1|$ is even, since it can be obtained by counting the (even)
number of ones in the rows of $F_1$ with indices in $X_1$, adding the
(even) number of ones in the columns of $F_1$ with indices in $X_2$,
and subtracting twice the number of ones in the intersection. As
$|\Omega_0|+|\Omega_1|=2 \lambda^2$, it follows that $|\Omega_0|$ is
also even.

Now let $p$ be the total number of pairs of ones in the
superposition of $F_1$ with the other $(k-1)$ squares. Each square
contributes $\lambda^2$ to $p$, so $p=(k-1)\lambda^2\=1\pmod{2}$.
However, each cell in $\Omega_0$ contributes an odd number of times to
$p$ and each cell in $\Omega_1$ contributes an even number, showing
that $p\=|\Omega_0|\=0\pmod{2}$. This contradiction completes the proof.
\end{proof}

In particular,~\tref{t:no2or3mod4} shows that $k=5$ is the
smallest $k$ for which~\tref{t:relmax} says anything, aside from the
fairly trivial case of $k=1$. This will be significant later, but for the
moment we just give an example to show that a relation can be achieved
when $k=5$. Consider the following 5-MOFS$(6)$, shown superimposed on the
left. Their $\Z_2$-sum is shown on the right, demonstrating that they
satisfy a $(5,3)$-relation and hence are maximal.
\begin{equation}\label{e:5maxMOFSrel}
\left[
\begin{array}{cccccc}
11011&10111&01100&00001&00010&11100\\
10100&01111&11011&00010&11100&00001\\
01111&11000&10111&11100&00001&00010\\
01001&10001&00101&10110&01110&11010\\
10010&00110&01010&01101&11001&10101\\
00100&01000&10000&11011&10111&01111
\end{array}
\right]
\qquad\qquad
\left[
\begin{array}{cccccc}
0&0&0&1&1&1\\
0&0&0&1&1&1\\
0&0&0&1&1&1\\
0&0&0&1&1&1\\
0&0&0&1&1&1\\
1&1&1&0&0&0
\end{array}
\right]
\end{equation}

We next show that even in the case when $k\=1\pmod{4}$ there is another
restriction on what relations are possible.

\begin{theorem}\label{t:restrictrel}
Let $\lambda$ be odd. Suppose that there exists $k$-{\rm MOFS}$(2\lambda)$ 
with a full relation $R=(X_1,\ldots, X_{k+2})$.
If $k\=1\pmod 8$, then $|X_1||X_2|=1\pmod{4}$ and
if $k\=5\pmod 8$, then $|X_1||X_2|=3\pmod{4}$.
\end{theorem}

\begin{proof}
Let $R=(X_1,\ldots, X_{k+2})$ be a full relation on a
$k$-MOFS$(2\lambda)$.  By \lref{l:parityrel}, $|X_1|$ and $|X_2|$ are
both odd.
Let $x_i$ be the number of $k$-tuples in the superposition
of the $k$-MOFS$(2\lambda)$ that contain exactly $i$ ones.  Let
$\alpha$ be the number of ones in the superposition of the
$k$-MOFS$(2\lambda)$ that lie in some position $(r,c)$ where $r \in X_1$
and $c \in X_2$.  As $R$ is a relation, $\alpha$ must be even.  Then
the total number of ones in the superposition of the
$k$-MOFS$(2\lambda)$ that lie in a position $(r,c)$ such that exactly
one of $r \in X_1$ and $c \in X_2$ is true is
\begin{equation}\label{e:oddixi}
\sum_{\text{odd }i}ix_i = \lambda k|X_1|-\alpha + \lambda k|X_2|-\alpha \= \lambda k(|X_1|+|X_2|)\pmod{4}.
\end{equation}
Note that equations 
(\ref{e:ixi}), (\ref{e:ic2xi}) and (\ref{e:oddxialt})
from the proof of~\tref{t:no2or3mod4} only depend on $\lambda$ being odd 
and so each of them are still valid in the current setting.
From~\eref{e:oddixi} and~\eref{e:oddxialt} we have
\begin{equation}\label{e:sumimod3}
\sum_{i \equiv 3 \text{\,(mod\,4)}}2x_i 
 \= \sum_{\text{odd }i}(i-1)x_i 
 \= \lambda (k-2)(|X_1|+|X_2|)+2|X_1||X_2|\pmod{4}.
\end{equation}
By simplifying 2$\times$(\ref{e:ic2xi})+(\ref{e:oddxialt})-(\ref{e:ixi}), 
one finds that
\begin{align}
k(k-3)\lambda^2 +2\lambda(|X_1|+|X_2|)-2|X_1||X_2|
&=\sum_{i=0}^{k}i(i-2)x_i + \sum_{\text{odd }i}x_i\nonumber\\
&\=\sum_{\text{odd }i}(i-1)^2x_i
\= \sum_{i \equiv 3\text{\,(mod\,4)}}4 x_i \pmod{8}.\label{e:sumimod3alt}
\end{align}
By setting $2\times$\eref{e:sumimod3}$=$\eref{e:sumimod3alt} 
and noting that
$\lambda$ is odd, we see that
\begin{equation}\label{e:lststp}
k(k-3)+2\lambda (3-k)(|X_1|+|X_2|)-6|X_1||X_2|\=0\pmod8.
\end{equation}
Since $k$ is odd, we see that $2\lambda(3-k)(|X_1|+|X_2|)\=0\pmod8$.
Thus if $k\=1\pmod8$, then we must have $|X_1||X_2|\=1\pmod{4}$
and  if $k\=5\pmod8$, then we must have $|X_1||X_2|\=3\pmod{4}$, by~\eref{e:lststp}.
\end{proof}

Having shown some restrictions on which relations are possible, our
next goal is to show that certain relations are actually
achievable. We first give a lemma characterising triples of MOFS.

\begin{lemma}\label{l:bsictriple}
For each triple $t$, let $x_t$ count the number of cells where $t$
occurs in the superimposition of three frequency squares $F_1$, $F_2$
and $F_3$ of order $2\lambda$. Then $F_1$, $F_2$ and $F_3$ are
orthogonal if and only if
\begin{equation}\label{e:tripMOFS}
  x_{000}=x_{011}=x_{101}=x_{110}\mbox{ \ and \ }
  x_{001}=x_{010}=x_{100}=x_{111}.
\end{equation}
\end{lemma}

\begin{proof}
  
  Orthogonality requires that
\begin{align*}
\lambda^2
&=x_{000}+x_{001}=x_{010}+x_{011}=x_{100}+x_{101}=x_{110}+x_{111}\\
&=x_{000}+x_{010}=x_{001}+x_{011}=x_{100}+x_{110}=x_{101}+x_{111}\\
&=x_{000}+x_{100}=x_{001}+x_{101}=x_{010}+x_{110}=x_{011}+x_{111},
\end{align*}
which is equivalent to \eref{e:tripMOFS}.
\end{proof}

We say that two
relations $(X_1,\dots,X_{k+2})$ and $(X'_1,\dots,X'_{k+2})$
are \emph{equivalent} if $|X_i|=|X'_i|$ for $1\le i\le k+2$.

\begin{theorem}\label{t:smallkrel}
Let $\Lambda=\{1,2,\dots,\lambda\}$. For $k\le3$ the following
relations are achieved, and every relation that is achieved is
equivalent to one of these:
\begin{itemize}
\item $k=1$: $X_1=X_2=\Lambda$.
\item
$k=2$: each of $\lambda$, $|X_1|$ and $|X_2|$ is even and either
$X_1=\Lambda$ or $X_2=\Lambda$.
\item
$k=3$: each of $\lambda$, $|X_1|$ and $|X_2|$ is even.
\end{itemize}
\end{theorem}

\begin{proof}
First suppose $k=1$, so we are looking for a relation on a single
frequency square $F$. Considering the fact that the first row of $F$ contains $\lambda$ zeroes and
$\lambda$ ones we deduce that $|X_2|=\lambda$. Similar consideration of
the first column of $F$ shows that $|X_1|=\lambda$. So without loss of
generality $X_1=X_2=\Lambda$.  Moreover, this is achieved by a
frequency square with four square blocks of order $\lambda$, in the
pattern given in~\eref{e:blocks}.

Next suppose that $k=2$. Let $\{F_1,F_2\}$ be a $2$-MOFS$(2\lambda)$
and $(X_1,X_2,X_3,X_4)$ be a relation on $F_1$ and $F_2$.
By orthogonality,  the superposition of $F_1$ and $F_2$
must contain $\lambda^2$ occurrences of each of the pairs $(0,1)$ and
$(1,0)$. It follows that
$(2\lambda-|X_1|)|X_2|+(2\lambda-|X_2|)|X_1|=2\lambda^2$,
which implies that $|X_1|=\lambda$ or $|X_2|=\lambda$.
Also, $|X_1|$ and $|X_2|$ must be even, by \lref{l:parityrel}.
Therefore, the conditions in the theorem for $k=2$ are necessary.
To show sufficiency, we construct the superimposition of $F_1$ and
$F_2$:
\[
\left[
\begin{array}{cccc}
\mathbf{00}&\mathbf{11}&\mathbf{01}&\mathbf{10}\\
\mathbf{11}&\mathbf{00}&\mathbf{10}&\mathbf{01}\\
\mathbf{01}&\mathbf{10}&\mathbf{00}&\mathbf{11}\\
\mathbf{10}&\mathbf{01}&\mathbf{11}&\mathbf{00}
\end{array}
\right],
\]
where the first two blocks on the diagonal have dimensions
$|X_1|/2\times|X_2|/2$, while the last two blocks on the diagonal
have dimensions $(\lambda-|X_1|/2)\times(\lambda-|X_2|/2)$.

Lastly, consider the case $k=3$. By~\tref{t:no2or3mod4} and \lref{l:parityrel}, we see that
$\lambda$, $|X_1|$ and $|X_2|$ must all be even.
Thus, the conditions in the theorem for $k=3$ are necessary.

Now we show sufficiency.
For even integers $y_1$ and $y_2$, let 
\[
\mathbf{B}_{y_1,y_2} = \left[
\begin{array}{cccc}
\mathbf{000}&\mathbf{011}&\mathbf{101}&\mathbf{110}\\
\mathbf{110}&\mathbf{101}&\mathbf{011}&\mathbf{000}\\
\mathbf{101}&\mathbf{110}&\mathbf{000}&\mathbf{011}\\
\mathbf{011}&\mathbf{000}&\mathbf{110}&\mathbf{101}
\end{array}
\right],
\]
where the first two blocks on the diagonal have dimensions
$\left\lceil \frac{y_1}{4}\right\rceil \times \left\lceil
\frac{y_2}{4}\right\rceil$ and the last two blocks on the diagonal
have dimensions $\left\lfloor \frac{y_1}{4}\right\rfloor \times
\left\lfloor \frac{y_2}{4}\right\rfloor$.  Notice that every $3$-tuple
has an even number of ones and the pairs $(0,0)$, $(0,1)$, $(1,0)$, $(1,1)$
occur the same number of times when each tuple is restricted to 2
entries.  The number of zeroes and ones in each row is balanced on the
2nd and 3rd entries and number of zeroes and ones in each column is
balanced on the 1st and 2nd entries.  For even integers $y_1$ and
$y_2$ let $\mathbf{B}^{C}_{y_1,y_2}$ be the complementary array to
$\mathbf{B}_{y_1,y_2}$.  That is,
\[
\mathbf{B}^C_{y_1,y_2} = 
\left[
\begin{array}{cccc}
\mathbf{111}&\mathbf{100}&\mathbf{010}&\mathbf{001}\\
\mathbf{001}&\mathbf{010}&\mathbf{100}&\mathbf{111}\\
\mathbf{010}&\mathbf{001}&\mathbf{111}&\mathbf{100}\\
\mathbf{100}&\mathbf{111}&\mathbf{001}&\mathbf{010}
\end{array}
\right],
\]
where the dimensions of the diagonal blocks are the same as those in $\mathbf{B}_{y_1,y_2}$.
For any even $x_1$, $x_2$ and $\lambda$ with $x_1,x_2 \leq 2\lambda$, 
we claim that the array  
\begin{equation}\label{e:relon3-MOFS}
\left[
\begin{array}{ll}
\mathbf{B}_{x_1,x_2}&\mathbf{B}^{C}_{x_1, 2\lambda-x_2}\\
\mathbf{B}^C_{2\lambda-x_1,x_2}&\mathbf{B}_{2\lambda -x_1, 2\lambda-x_2}
\end{array}
\right].
\end{equation}
is the superposition of $3$-MOFS$(2\lambda)$ that 
has a relation on $X_1=\{1,\dots,x_1\}$ and $X_2=\{1,\dots,x_2\}$.
The fact that $x_2\equiv2\lambda-x_2\pmod{4}$ ensures balance in the rows
of the first square in~\eref{e:relon3-MOFS}.
The fact that $x_1\equiv2\lambda-x_1\pmod{4}$ ensures balance in the columns
of the third square in~\eref{e:relon3-MOFS}. Other aspects of our claim
are straightforward to check, using \lref{l:bsictriple}. 
\end{proof}

It is worth noting that~\tref{t:relmax} does not generalise to even
$\lambda$.  For example, if we take $x_1=x_2=\lambda\equiv0\pmod{4}$,
then the MOFS in \eref{e:relon3-MOFS} satisfy a relation but are not
maximal, because they are orthogonal to the frequency square with the
block structure given in~\eref{e:blocks}.

To finish this section we consider the relations satisfied by complete
sets.  A $(v_*,k_*,\lambda_*)$-design is a collection ${\mathcal B}$
of $k_*$-subsets of a set $V$ of size $v_*$ such that each pair from
$V$ is contained in exactly $\lambda_*$ blocks (we are using $_*$
subscripts on variables here to distinguish them from similarly named
variables used throughout the paper). Such a design is said to be
\emph{resolvable} if there is a partition of the blocks into
\emph{parallel classes} (i.e. partitions of $V$).

Our next result is implied by Theorem 3.5 from Jungnickel, et
al~\cite{JMM}, and the proof we give is basically the same as in that
paper. We include a proof for completeness and because we want to
squeeze a tiny bit more out of it.

\begin{theorem} \label{t:compldesign}
Let ${\mathcal F}=\{F_1, F_2, \dots ,F_k\}$ be complete {\rm MOFS} 
of order $n$ where $k=(n-1)^2$.  Construct a multiset of blocks
${\mathcal B}$ where for each $s\in\{0,1\}$ and $1\le i\le k$ there
is a block $B_{s,i}\in {\mathcal B}$ that is the set
of columns containing the entry $s$ in the first row of $F_i$.  
Then ${\mathcal B}$ is a resolvable $(n,n/2,(n-1)(n-2)/2)$-design.
Also, if $n\equiv0\pmod4$ then ${\mathcal F}$ satisfies a $(0,n)$-relation. 
\end{theorem}

\begin{proof}
Without loss of generality, consider a cell $(1,1)$. For each cell
$(r,c)\neq (1,1)$, let $\theta_{r,c}$ be the number of frequency
squares in ${\mathcal F}$ such that cells $(1,1)$ and $(r,c)$ contain
the same entry.  Next, define
$$\alpha_1=\sum_{c=2}^n \theta_{1,c}+ \sum_{r=2}^n \theta_{r,1}\quad\text{and}
\quad\alpha_2=\sum_{r=2}^n\sum_{c=2}^n \theta_{r,c}.$$ 
In each frequency
square, there are $n/2-1$ cells other than $(1,1)$ in row (column)
$1$ containing the same entry as cell $(1,1)$.  Therefore,
$\alpha_1=2k(n/2-1)$.  Similarly, $\alpha_2=k(n^2/2-n+1)$.

Let $F,F'\in {\mathcal F}$.  Then since $F$ and $F'$ are orthogonal,
the number of cells $(r,c)\neq (1,1)$ such that $(r,c)$ and $(1,1)$
share the same entry in $F$ and $F'$ is equal to $n^2/4-1$.
Thus,
$$\sum_{(r,c)\neq (1,1)} {\theta_{r,c} \choose 2} = {k\choose 2}(n^2/4-1).$$
It follows that:
\begin{equation}\label{e:lamsq}
  \sum_{(r,c)\neq (1,1)} \theta_{r,c}^2
= k(k-1)(n^2/4-1)+\sum_{(r,c)\neq (1,1)}\theta_{r,c}
= k(k-1)(n^2/4-1)+\alpha_1+\alpha_2.
\end{equation}

Next, define $\theta_1=(n/2-1)(n-1)$ and $\theta_2=(n^2/2-n+1)$ and observe that
\begin{align*}
\sum_{c=2}^n (\theta_1-\theta_{1,c})^2 +& \sum_{r=2}^n (\theta_1-\theta_{r,1})^2
+\sum_{r=2}^n\sum_{c=2}^n (\theta_2-\theta_{r,c})^2\\
&=2(n-1)\theta_1^2+(n-1)^2\theta_2^2
-2\alpha_1\theta_1-2\alpha_2\theta_2
+\sum_{(r,c)\neq (1,1)} \theta_{r,c}^2=0,
\end{align*}
given \eref{e:lamsq} and the fact that $k=(n-1)^2$.  Thus, we must
have $\theta_{1,c}=\theta_1=\theta_{r,1}$ for each $r\neq 1$ and
$c\neq 1$, and $\theta_{r,c}=\theta_2$ for all $r\ge2$ and $c\ge2$.
In particular, for each column $c\neq 1$, the number of frequency
squares which contain the same entry in both cells $(1,1)$ and $(1,c)$
is constant and equal to $\theta_1$. The same is true if we replace
$(1,1)$ by any other fixed cell in the first row.  Thus in the set
${\mathcal B}$ as defined above, each pair of columns occurs in
precisely $\theta_1$ blocks.  Also the blocks
$B_{0,i}$ and $B_{1,i}$ are complementary sets, for each $i$.
It follows that ${\mathcal B}$ is a
resolvable $(n,n/2,(n-1)(n-2)/2)$-design.

Finally, note that if $n\equiv0\pmod4$ then
$(n-1)^2\equiv\theta_1\equiv\theta_2\equiv1\pmod2$. Hence, if we
standardise ${\mathcal F}$ by complementing any square that has a zero
in cell $(1,1)$ then the $\Z_2$-sum of ${\mathcal F}$ will be a
matrix in which every entry is $1$.
\end{proof}

The method used in \tref{t:compldesign} would also show that if
$n\equiv2\pmod4$ then ${\mathcal F}$ must satisfy a relation. However,
this is a moot point, given \Cref{cy:nooddcomp} below.  The following
theorem is stated in \cite[Thm~II.7.31]{CRC}; again we include a 
short proof for completeness.

\begin{theorem}\label{t:nodes} 
  For odd positive integers $t\ge1$ and $k>1$, there does not exist a
  resolvable $(2k,k,t(k-1))$-design.
\end{theorem}

\begin{proof}
Suppose such a design ${\mathcal B}$ exists and let $\lambda_*=t(k-1)$. 
For any subset $W$ of $V$, let $r_{W}$ be the number of blocks
containing $W$ as a subset.  Suppose that $r_{\{x,y,z\}}=s$ for some
distinct $x$, $y$ and $z$. Then the number of blocks containing $x$
but neither $y$ nor $z$ is $r_{\{x\}}-2\lambda_*+s=t+s$. But for each such
block $B$, the pair $\{y,z\}$ must be in the other block of the
parallel class containing $B$. Thus $\lambda_*=r_{\{y,z\}}=t+2s$. But
$\lambda_*$ is even, a contradiction.
\end{proof}

\begin{corollary}\label{cy:nooddcomp}
There does not exist a complete {\rm MOFS} of order $n$ whenever $n/2$
is odd and $n>2$.
\end{corollary}

\begin{proof}
Taking $k=n/2$ and $t=n-1$ in \tref{t:nodes}, we find that the design
required by \tref{t:compldesign} does not exist when $n/2$ is odd.
\end{proof}

\section{The lonely bachelor}\label{s:bach}

In this section we show that for order $n=2\lambda$, there are no bachelor
frequency squares if $\lambda$ is even and only one bachelor square
(up to isomorphism) if $\lambda$ is odd.

Let $A_{2\lambda}$ be the unique frequency square satisfying
a (non-trivial) relation, as shown in~\tref{t:smallkrel}.
We will make use of the following well known corollary of Dirac's Theorem:

\begin{theorem}\label{t:perfect}
Let $G$ be a simple graph with $2\lambda$ vertices and minimum
degree at least $\lambda$.  Then $G$ has a perfect matching.
\end{theorem}

\begin{theorem}\label{t:bach}
Let $B$ be a frequency square of type $F(2\lambda;\lambda)$.  Then $B$
is a bachelor if and only if $\lambda$ is odd and $B$ is isomorphic to
$A_{2\lambda}$.
\end{theorem}

\begin{proof}
The fact that $A_{2\lambda}$ is a bachelor when $\lambda$ is odd
follows from \tref{t:relmax}.  
So it suffices to construct an orthogonal mate $B'$ of type
$F(2\lambda;\lambda)$ for any $B$ such that $\lambda$ is even or $B$
is not isomorphic to $A_{2\lambda}$

With respect to any two distinct rows $r$ and $r'$ of $B$, we say that
a column $c$ is of type $1$ if cells $(r,c)$ and $(r',c)$ contain the
same entry; otherwise column $c$ is of type $2$. We say that a pair of
distinct rows $\{r,r'\}$ in $B$ is \emph{bad} if every column is of
type $2$ with respect to that pair.  

We aim to partition the rows of $B$ into pairs so that no pair is bad.
Observe that for a given row $r$, there are at most $\lambda$ rows
$r'$ such that $\{r,r'\}$ is a bad pair.  Suppose first that there
exists a row $r$ such that there are exactly $\lambda$ rows $r'$ for
which $\{r,r'\}$ is a bad pair. Each of those $\lambda$ rows must
be identical, and it quickly follows that $B$ is isomorphic to
$A_{2\lambda}$.  By our assumptions, $\lambda$ must then be even, so
we can easily avoid bad pairs by partitioning rows into pairs
of identical rows.

Otherwise for each row $r$ there exists \emph{at most} $\lambda-1$ rows
$r'$ such that $\{r,r'\}$ is a bad pair. Form a graph $G$ where the
vertices are the $2\lambda$ rows of $B$, and two rows $r$
and $r'$ are joined by an edge if and only if $\{r,r'\}$ is not a bad
pair. The minimum degree of $G$ is at least $\lambda$, so by 
\tref{t:perfect}, the graph $G$ contains a perfect matching. Thus there
exists a partition ${\mathcal P}$ of the rows of $B$ into pairs, none
of which is bad.

For each $\{r,r'\}\in {\mathcal P}$, we next construct corresponding
rows $r$ and $r'$ in $B'$ so that:
\begin{itemize}
\item Rows $r$ and $r'$ in $B'$ each contain $\lambda$ ones and $\lambda$ zeroes; 
\item Each column in $B'$ is of type $2$ with respect to rows $r$ and $r'$; and 
\item When rows $r$ and $r'$ of $B$ and $B'$ are superimposed, 
      the ordered pairs $(0,0)$, $(0,1)$ and $(1,0)$ and $(1,1)$ each occur $\lambda$ times.  
\end{itemize}
Assuming these properties hold for every pair in ${\mathcal P}$, the
first and second conditions guarantee that $B'$ is a frequency square
of type $F(2\lambda;\lambda)$ while the third condition guarantees
that $B'$ is an orthogonal mate for $B$.

Hence, given $\{r,r'\}\in {\mathcal P}$ it remains to determine the
entries of rows $r$ and $r'$ of $B'$ satisfying these properties.  We
say that a column $c$ is of type $1$a (respectively, $1$b) with
respect to $(r,r')$ if $(r,c)$ and $(r',c)$ each contain $0$
(respectively, $1$).  We say that a column $c$ is of type $2$a
(respectively, $2$b) with respect to $(r,r')$ if $(r,c)$ contains $0$
(respectively, $1$) and $(r',c)$ contains $1$ (respectively, $0$).

Within rows $r$ and $r'$ of $B$, let $t_{1a}$ be the number of columns
of type $1a$, with $t_{1b}$, $t_{2a}$ and $t_{2b}$ defined
similarly. Since each row contains $\lambda$ zeroes and $\lambda$ ones,
$$t_{1a}+t_{2a}=t_{1b}+t_{2b}=t_{1b}+t_{2a}=t_{1a}+t_{2b}=\lambda,$$
from which it follows that $t_{1a}=t_{1b}$ and $t_{2a}=t_{2b}$. 
Since the pair of rows $\{r,r'\}$ is not bad, $t_{1a}=t_{1b}>0$. 
Of the columns in $B$ of types $1a$, $1b$, $2a$ and $2b$, respectively,
we place $\lambda-\lfloor t_{1a}/2\rfloor-2\lceil t_{2a}/2\rceil$, 
$\lfloor t_{1a}/2\rfloor$, $\lceil t_{2a}/2\rceil$ and $\lceil t_{2a}/2\rceil$ 
columns of type $2a$ in the corresponding positions in $B'$. That gives
us $\lambda$ columns of type $2a$, and the
other $\lambda$ columns in $B'$ are made to be of type $2b$. Note that 
\begin{align*}
t_{1a}\ge t_{1a}-\lfloor t_{1a}/2\rfloor
=\lambda-\lfloor t_{1a}/2\rfloor-t_{2a}
&\ge\lambda-\lfloor t_{1a}/2\rfloor-2\lceil t_{2a}/2\rceil\\
&\ge\lambda-\lfloor t_{1a}/2\rfloor-t_{2a}-1=t_{1a}-\lfloor t_{1a}/2\rfloor-1\ge0
\end{align*}
given that $t_{1a}\ge1$. 
In particular, $t_{1a}\ge
\lambda-\lfloor t_{1a}/2\rfloor-2\lceil t_{2a}/2\rceil\ge0$, which shows that
our construction is feasible. Moreover, the
ordered pairs $(0,0)$, $(0,1)$ and $(1,0)$ and $(1,1)$ each occur 
$t_{1a}+\lceil t_{2a}/2\rceil+t_{2a}-\lceil t_{2a}/2\rceil=\lambda$
times in rows $\{r,r'\}$ of the superposition of $B$ and $B'$,
as required.
\end{proof}

\section{Trades in MOFS}\label{s:trades}

In this section we consider some transformations that can be used to
alter the structure of a set of MOFS. The idea is to identify a
comparatively small number of cells that can be changed, whilst
preserving the property of being a set of MOFS.  This leads to the
idea of \emph{trades}, which has been extensively studied for other
designs~\cite{Bil,Cav,HK}, but we are not aware of any previous work
regarding trades in binary MOFS.

Formally we define a trade in a set $\{F_1,\dots,F_k\}$ of MOFS to
be a suitable set of cells $C_i$ for each $F_i$ in the set of MOFS.
An individual $C_i$ can be empty, but they should not all be empty.  
To \emph{switch} the trade we change the entries in every cell in $C_i$
in square $F_i$ for each $i$. The test for whether the chosen cells
are suitable is that the result of switching on the trade should 
again be a set of MOFS. We do not attempt to characterise the general case
any further, but instead look at a simple special case which is
already powerful enough to be interesting.  In this special case
the nonempty $C_i$ are all equal.

\begin{theorem}\label{t:basictrades}
Suppose that we have a set $\F=\{F_1,F_2,\dots,F_k\}$ of MOFS of order $n$ 
and $\emptyset\neq C\subseteq N(n)\times N(n)$. For $1\le i\le k$, let 
$C_i=C$ if $F_i$ either agrees with $F_1$ on every cell in
$C$ or disagrees with $F_1$ on every cell in $C$, and let $C_i=\emptyset$
otherwise. 
Let $V_{i,a}=\{(x,y)\in C:F_i(x,y)=a\}$ for $1\le i\le k$ and $a\in\{0,1\}$.
Then $T=(C_1,C_2,\dots,C_k)$ forms a trade if and only if
\begin{itemize}
\item 
Each row or column of $F_1$ contains equal numbers of 
zeros and ones within the cells in $C$.
\item 
For each $j$ such that $C_j=\emptyset$, we have
$|V_{1,1}\cap V_{j,1}|=|V_{1,0}\cap V_{j,1}|$.
\end{itemize}
\end{theorem}

\begin{proof}
We note that $T$ is nonempty, because $C_1=C\neq\emptyset$.  Let
$\{F'_1,F'_2,\dots,F'_k\}$ be the matrices produced by switching
$\{F_1,F_2,\dots,F_k\}$ on $T$.  First suppose that the two conditions
are satisfied.  The first condition guarantees that each $F'_i$ is a
frequency square.  The two conditions together imply that
$|V_{1,1}\cap V_{j,0}|=|V_{1,0}\cap V_{j,0}|$.  Therefore, the two
conditions ensure that $F'_i$ is orthogonal to $F'_j$ for all $1\le
i<j\le k$, since $F_i$ is orthogonal to $F_j$. Thus, $T$ is a trade if
the two conditions are satisfied.  Conversely, if $T$ is a trade, then
the first condition is satisfied, since $F_1'$ is a frequency
square. Finally, for $j$ such that $C_j = \emptyset$, $F_1'$ is
orthogonal to $F_j'=F_j$ only if the second condition is
satisfied. This completes the proof.
\end{proof}

We call the trades described in~\tref{t:basictrades}, \emph{basic trades}.
There are some fairly trivial examples of basic trades where switching
on the trade does not change the combinatorial structure of the MOFS.

\begin{lemma}\label{l:trivtrades}
Below are basic trades for which switching does not change the isomorphism class of the MOFS.
\begin{itemize}
\item $C=N(n)\times N(n)$.
\item Taking $C$ to be the set of cells on which $F_1$    agrees with $F_j$ for some fixed $j>1$. 
\item Taking $C$ to be the set of cells on which $F_1$ disagrees with $F_j$ for some fixed $j>1$. 
\end{itemize}
\end{lemma}

\begin{proof}
If we take $C=N(n)\times N(n)$, then $C_1=C$. However, $C_i=\emptyset$
for $i>1$ since $F_i$ cannot agree with $F_1$ on every cell or
disagree with it on every cell, since $F_i$ is orthogonal to $F_1$.
So, in this case the trade simply complements $F_1$ (interchanges
zeros and ones within $F_1$).

Next consider what happens when we take $C$ to be the set of cells on
which $F_1$ disagrees with $F_j$. Since $F_1$ is orthogonal to $F_j$,
$C$ must consist of exactly half of all cells in these squares.
Switching on these cells converts $F_1$ into $F_j$ and vice versa.
Let $i\in N(k)\setminus\{1,j\}$. We know that $F_i$ is orthogonal to
$F_1$ and hence agrees with $F_1$ on exactly half of its cells.  If
$F_i$ disagrees with $F_1$ on every cell in $C$ it would have to equal
$F_j$. Also, if $F_i$ agrees with $F_1$ on every cell in $C$ it would
have to disagree with $F_j$ in every cell. Either option is
impossible, since $F_i$ is orthogonal to $F_j$. So we conclude that
$C_i=\emptyset$. Moreover, \lref{l:bsictriple} ensures that
$|V_{1,1}\cap V_{i,1}|=|V_{1,0}\cap V_{i,1}|$, so the trade is valid, by
\tref{t:basictrades}.
The result of switching on the trade is to interchange $F_1$ and
$F_j$, resulting in an isomorphic set of MOFS.

Taking $C$ to be the set of cells on which $F_1$ agrees with $F_j$
works similarly. Switching on it is equivalent to interchanging $F_1$
and $F_j$ and then complementing both squares.
\end{proof}

\begin{theorem}\label{t:tradesinHadset}
Let $n>2$ and let $\F$ be a set of {\rm MOFS}$(n)$ in which every frequency
square has the property that every pair of rows is either equal or
complementary.  Then $\F$ has at least $(n/2)^4$ basic trades each of which 
produce a new set of {\rm MOFS} that are not isomorphic to $\F$.
\end{theorem}

\begin{proof}
Up to permutations of the rows and columns each square in $\F$ has
block structure~\eref{e:blocks}. Choose $C$ to be any of the
``intercalates'' in $F_1$, that is, $2\times2$ submatrices that meet
all four blocks in $F_1$. There are $(n/2)^4$ choices for $C$. We
claim each of them produces a basic trade. It is obvious that each row
and column of $F_1$ has the same number of zeros and ones in cells in
$C$. Moreover, in each $F_i$ the number of ones that occur in cells in
$C$ must be even, since the two rows that meet $C$ are either equal or
complementary. If $C$ induces an identity matrix or its complement in
$F_i$, then $C_i=C$.  In the other squares, $C$ must induce a matrix
with either constant rows or constant columns, and
$C_i=\emptyset$. Either case satisfies the second condition
in~\tref{t:basictrades}.

After switching on the trade, $F_1$ becomes a matrix with two rows
that are neither equal nor complementary, so the new set of MOFS is
not isomorphic to the original set.
\end{proof}

\begin{theorem}\label{t:manyMOFS}
For any $n$ for which there exists a Hadamard matrix of order $n$
there exists at least $2^{n^2/4-O(n\log n)}$ isomorphism classes of 
complete {\rm MOFS}$(n)$.
\end{theorem}

\begin{proof}
By~\tref{t:Hadamard}, from a Hadamard matrix we can construct a
complete MOFS of the same order. The construction used to prove
that theorem ensures that every square in that set has the property
that each pair of rows is equal or complementary.  Any such square is
determined by its first row and first column. We assume, without loss
of generality, that the first square has block structure
\eref{e:blocks}.
 
Say that a complete set $\{F_1,\dots,F_{(n-1)^2}\}$ of MOFS$(n)$ is
\emph{standardised} if 
$F_i(1,1)=0$ for $1\le i\le(n-1)^2$.
An isomorphism class of MOFS contains $e^{O(n\log n)}$ standardised MOFS
since the rows and columns can be permuted in $n!^2=e^{O(n\log n)}$ ways
and then there is a unique way to standardise the set by complementing
any squares that have a one in their $(1,1)$ cell.

We make use of the basic trades described in~\tref{t:tradesinHadset}.
Consider the set of trades $\T=\{T_{r,c}:2\le r\le n/2,\,2\le c\le
n/2\}$, where $T_{r,c}$ uses the cells
$$C=\big\{(r,c),(r+n/2,c),(r,c+n/2),(r+n/2,c+n/2)\big\}.$$ 
Note that $T_{r,c}$ and $T_{r',c'}$
are on disjoint sets of cells unless $r=r'$ and $c=c'$. Hence we can
switch on any subset of $\T$ to obtain a new complete set of MOFS.  As
we have preserved the first row and column of each square, and these
are unique to that square, we cannot produce two sets of MOFS that are
the same but have their squares in a different order. Also, every set
that we produce is standardised, so we have built
$2^{|\T|}=2^{n^2/4-O(n)}$ standardised complete MOFS. The
result now follows since the number of isomorphism classes is at least
$2^{n^2/4-O(n)}/e^{O(n\log n)}=2^{n^2/4-O(n\log n)}$.
\end{proof}

\section{Computational results}\label{s:comp}

In this section we report the results of a computational exploration
of maximality among sets of MOFS of small orders. Our results were each
obtained by two independently written programs. The computations took
several months of CPU time. In order to present sets of MOFS more compactly
we adapt the notation used earlier. Rather than just
superimposing the squares as we did in~\eref{e:5maxMOFSrel}, we
superimpose them and then convert the resulting entries from binary
to decimal. For example, the first row of~\eref{e:5maxMOFSrel}
would be written as [27,23,12,1,2,28] rather than
[11011,10111,01100,00001,00010,11100].

All MOFS$(4)$ extend to a complete set, so there are no maximal
MOFS$(4)$ that are not complete. There are three different complete
MOFS$(4)$, up to isomorphism \cite{CRC}.  One of the 3 sets is related
to the other two by basic trades that switch 4 cells in 4 squares:
\begin{equation}\label{e:tradMOFS4}
\left[
\begin{array}{cccc}
511&448&21&42\\
76&115&410&421\\
259&316&\mk233&\mk214\\
176&143&\mk358&\mk345
\end{array}
\right]
\qquad\qquad
\left[
\begin{array}{cccc}
511&448&21&42\\
76&115&\mk410&\mk421\\
259&316&233&214\\
176&143&\mk358&\mk345
\end{array}
\right]
\end{equation}
A basic trade on the cells highlighted in the left hand copy above
changes the last four frequency squares in the set.  A basic trade on
the cells highlighted in the right hand copy above changes the 4th,
5th, 6th and 7th frequency squares in the set.  Both basic trades
produce MOFS that are not isomorphic to each other or the initial set
of MOFS. There is no single basic trade that switches between the
isomorphism classes of the MOFS that result from the two basic trades
highlighted in~\eref{e:tradMOFS4}. The MOFS produced by the trade
shown on the right in \eref{e:tradMOFS4} are the ones produced from
\tref{t:Hadamard}. They are the only ones in which every frequency
square consists of rows which pairwise are either equal or complementary.

Up to equivalence, there are 6 frequency squares of type $F(6;3)$ and
2435 pairs of MOFS of the same type. For each pair $P$ we found and
stored every ``mate'' that allows the pair to extend to a triple. The
number of mates ranged from 5937 to 7413. A graph $\Gamma_{\!P}$ was
then constructed with the mates as its vertices, and edges indicating
orthogonality.  It was easily observed that every vertex in
$\Gamma_{\!P}$ had positive degree (indeed, the minimum degree ranged
from 548 to 1369) and that every edge was in a triangle. It follows
that aside from the unique bachelor (see Theorem \ref{t:bach}), there
are no maximal MOFS of type $F(6;3)$ containing fewer than 5
squares. There are a large number of maximal sets of 5 squares. By the
above method we generated $577\,418\,387$ (respectively, $1475$) 
$5$-maxMOFS$(6)$ that do (respectively, do not) satisfy a relation.
We did not store the former so we cannot say how many isomorphism
classes they represent. However, we did store the 1475 examples of
$5$-maxMOFS$(6)$ that do not satisfy a relation, and these come from
130 isomorphism classes.  The most symmetric $5$-maxMOFS$(6)$ has an
automorphism group of order 10. A representative of that class
follows. Its $\Z_2$-sum is the identity matrix.
\begin{equation}\label{e:5maxMOFSnorel}
\left[
\begin{array}{cccccc}
31&17&18&0&15&12\\
17&21&30&6&9&10\\
18&30&8&29&3&5\\
0&6&29&7&24&27\\
15&9&3&24&22&20\\
12&10&5&27&20&19
\end{array}
\right]
\end{equation}
One of the $5$-maxMOFS$(6)$ satisfying a relation was given
in~\eref{e:5maxMOFSrel}.

Next we used an elementary backtracking search to locate all cliques
of size 15 or more in~$\Gamma_{\!P}$. 
Each $k$-clique of $\Gamma_{\!P}$ corresponds to $(k+2)$-MOFS$(6)$. 
Of the 2435 pairs of MOFS$(6)$, there were 842 pairs that extended to 
a $17$-MOFS$(6)$ and no pair extended further. 
We conclude that the largest set of MOFS$(6)$ has cardinality 17. 
There are 18 sets of $17$-MOFS$(6)$ up to isomorphism. 
We now present these 18 sets, starting with this example:
\[
\left[
\begin{array}{cccccc}
72128& \MK91655& 44068& \MK53560& \mk731& \mk131071\\ 
115574& 15266& 58249& 15454& 87221& 101449\\ 
46877& 54474& \mk129267& \MK84231& \MK76344& \mk2020\\ 
107666& 39541& 86604& 28625& 52654& 78123\\ 
24107& 100765& \mk1395& 107246& \mk130880& 28820\\ 
26861& \MK91512& 73630& 104097& \MK45383& 51730
\end{array}
\right]
\]
The 6 lightly shaded cells indicate a trade which changes the first 6
squares in the set, and the 6 darkly shaded cells indicate a trade
which changes the last 6 squares in the set.  Both trades result in
non-isomorphic $17$-MOFS$(6)$.

Similarly, the following matrices represent $17$-MOFS$(6)$ where
the shading shows 6 cells where there is a trade that changes the
first 6 frequency squares and leaves the other 11 unchanged.
Each such trade leads to a non-isomorphic $17$-MOFS$(6)$.
\[
\left[
\begin{array}{cccccc}
94304&36518&26456&37249&67615&131071\\
60811&119066&11061&13558&120429&68288\\
\mk107059&80348&130181&\mk24122&1513&49990\\
\mk23468&25671&37595&117109&81666&\mk107704\\
71639&47721&82094&\mk108300&59890&\mk23569\\
35932&83889&105826&92875&62100&12591
\end{array}
\right]
\]
\[
\left[
\begin{array}{cccccc}
7199&11745&123152&98958&21088&131071\\
110272&119623&6837&38266&92315&25900\\
61612&88540&\mk79467&\mk52305&110390&899\\
120753&41015&47562&92781&2012&89090\\
67942&\mk79384&\mk53005&31666&62663&98553\\
25435&\mk52906&83190&\mk79237&104745&47700
\end{array}
\right]
\]
\[
\left[
\begin{array}{cccccc}
25649&117472&70400&41438&7183&131071\\
47762&83118&59717&70263&\mk54648&\mk77705\\
89554&13765&124446&\mk54185&\mk76409&34854\\
110765&31324&18875&\mk77154&100244&\mk54851\\
36716&104731&14058&20117&123075&94516\\
82767&42803&105717&130056&31654&216
\end{array}
\right]
\]

In the next example the $17$-MOFS$(6)$ has two trades which lead to
non-isomorphic $17$-MOFS$(6)$. The first trade consists of switching the
first and last rows of the first square in the set. The second trade
consists of switching the first and last columns of the first square
in the set.
\[
\left[
\begin{array}{cccccc}
105968&128645&75835&27623&53258&1884\\
69126&32895&29592&98689&131071&31840\\
94557&12170&50789&122426&73952&39319\\
21283&124242&105164&19673&43828&79023\\
59598&72553&83382&46380&5843&125457\\
42681&22708&48451&78422&85261&115690
\end{array}
\right]
\]

Similarly, in the following two examples there is a trade consisting of
switching the first and last rows of the first square in the set:
\[
\left[
\begin{array}{cccccc}
125836&97351&66937&6902&46890&49297\\
32895&100227&31152&26176&71692&131071\\
8097&127032&41678&56607&91893&67906\\
113233&21466&117926&82733&11419&46436\\
87274&36404&78615&110040&53701&27179\\
25878&10733&56905&110755&117618&71324
\end{array}
\right]
\]
\[
\left[
\begin{array}{cccccc}
96320&74526&120239&6329&53090&42709\\
32895&126601&8084&119634&76268&29731\\
26544&103540&43595&77095&86683&55756\\
98689&19706&96885&18253&45734&113946\\
131071&39713&66754&59542&13657&82476\\
7694&29127&57656&112360&117781&68595\\
\end{array}
\right]
\]
Finally, we present two more $17$-MOFS$(6)$ which have no basic trades other
than those of the type covered by~\lref{l:trivtrades}.
\[
\left[
\begin{array}{cccccc}
65567&67553&62976&51608&14438&131071\\
89932&37271&10972&95538&42667&116833\\
47401&53876&109907&90797&87258&3974\\
107764&27950&17275&7361&129941&102922\\
29635&125130&103852&40575&68400&25621\\
52914&81433&88231&107334&50509&12792
\end{array}
\right]
\]
\[
\left[
\begin{array}{cccccc}
65567&51334&30488&47201&67552&131071\\
97457&26213&41298&119452&40399&68394\\
123843&81164&39861&1663&28842&117840\\
55078&19417&88171&107913&109108&13526\\
40520&111354&77447&19890&86357&57645\\
10748&103731&115948&97094&60955&4737
\end{array}
\right]
\]
This completes the specification of the 18 isomorphism classes of
$17$-MOFS$(6)$. No pair of these classes is connected by basic
trades unless our description specified such a relationship.

All 18 sets of $17$-MOFS$(6)$ satisfy a $(3,3)$-relation,
thereby demonstrating that they are maximal by~\tref{t:relmax}. 
We next consider the relations satisfied by $k$-MOFS$(6)$ for
$1<k<17$. By~\tref{t:no2or3mod4}, we need only consider $k\in\{5,9,13\}$.
Also, by~\tref{t:restrictrel}, when $k\in\{5,13\}$ 
we only need to consider $(r,s)$-relations where 
\begin{equation}\label{e:reltype5mod8}
(r,s)\in\{(1,3),(3,1),(3,5),(5,3)\}.
\end{equation}
A $k$-MOFS$(6)$ having any of the
relations in~\eref{e:reltype5mod8} can easily be converted into an isomorphic
set satisfying any of the the other 3 kinds of relations. Transposing all
the squares in a set of MOFS interchanges $r$ and $s$, and complementing
both $X_1$ and $X_3$ transforms a set of MOFS with an $(r,s)$-relation into 
one with an $(n-r,n-s)$-relation. Thus the only question is whether
there exists a set with any of the relations in~\eref{e:reltype5mod8} or not. 
We have already demonstrated a $5$-maxMOFS$(6)$ 
satisfying a $(5,3)$-relation in~\eref{e:5maxMOFSrel}.
Similarly, here are $13$-maxMOFS$(6)$ satisfying a
$(5,3)$-relation. The shaded cells indicate a basic trade which can be
switched to reach $13$-maxMOFS$(6)$ that do not satisfy any relation.
\begin{equation}\label{e:13maxMOFSnorel}
\left[
\begin{array}{cccccc}
\mk4095&4196&\mk4539&\mk3587&\mk7708&448\\
1576&2266&7495&\mk4847&4881&\mk3508\\
2181&7289&1910&5266&\mk1001&\mk6926\\
6442&1923&6832&2397&1126&5853\\
\mk5457&\mk798&3784&5548&2231&6755\\
4822&\mk8101&\mk13&2928&7626&1083
\end{array}
\right]
\end{equation}
Switching the trade only changes the first square in the set of MOFS.

For $9$-MOFS$(6)$ there were more possibilities, {\it a priori}.
By~\tref{t:restrictrel}, we need to consider $(r,s)$-relations
where 
\begin{equation}\label{e:reltype1mod8}
(r,s)\in\{(1,1),(5,5)\}\cup\{(1,5),(5,1)\}\cup\{(3,3)\}.
\end{equation}
Here we have partitioned the possibilities into sets of relations that
can be transformed into each by the moves described above.  As already
noted, only the last possibility is achieved by $17$-MOFS$(6)$.
However, there are $9$-MOFS$(6)$ achieving all of the options
in~\eref{e:reltype1mod8}.  We start by giving $9$-MOFS$(6)$ with a
$(1,1)$-relation then $9$-MOFS$(6)$ with a $(1,5)$-relation:
\[
\left[
\begin{array}{cccccc}
284&511&259&4&224&251\\
433&288&335&126&154&197\\
206&338&444&483&9&53\\
457&108&113&27&438&390\\
55&135&170&464&365&344\\
98&153&212&429&343&298
\end{array}
\right]
\qquad\qquad
\left[
\begin{array}{cccccc}
257&270&18&228&249&511\\
333&419&127&148&194&312\\
436&220&481&11&47&338\\
110&471&21&426&409&96\\
179&56&458&375&324&141\\
218&97&428&345&310&135
\end{array}
\right]
\]
Next we give $9$-maxMOFS$(6)$ that satisfy a $(3,3)$-relation. 
\begin{equation}\label{e:9maxMOFSnorel}
\left[
\begin{array}{cccccc}
449&106&180&307&93&398\\
180&449&106&398&307&93\\
106&180&449&93&398&307\\
511&7&280&169&210&356\\
280&511&7&356&169&210\\
7&280&511&210&356&169
\end{array}
\right]
\end{equation}
If the first and last rows of the first square
in~\eref{e:9maxMOFSnorel} are switched, then the result is a
$9$-maxMOFS$(6)$ that do not satisfy any relation.

In~\eref{e:5maxMOFSnorel},~\eref{e:9maxMOFSnorel}
and~\eref{e:13maxMOFSnorel} we have described $k$-maxMOFS$(6)$ for
$k\in\{5,9,13\}$ that do not satisfy a relation.  In fact we found
$k$-maxMOFS$(6)$ for all $5\le k\le 15$ that do not satisfy a
relation.  There are no $16$-maxMOFS$(6)$.  Here we give
$k$-maxMOFS$(6)$ for $k\in\{6,7,8,10,11,12,14,15\}$:
\[
\begin{array}{ccc}
\left[
\begin{array}{cccccc}
63&33&34&1&28&30\\
36&45&50&14&19&25\\
43&54&25&60&0&7\\
8&10&61&7&54&49\\
23&21&4&56&43&42\\
16&26&15&51&45&36
\end{array}\right]
&
\left[
\begin{array}{cccccc}
127&64&67&5&56&62\\
73&84&106&31&36&51\\
86&111&53&112&11&8\\
25&26&124&6&99&101\\
32&45&14&121&87&82\\
38&51&17&106&92&77
\end{array}\right]
&\left[
\begin{array}{cccccc}
255&128&129&14&114&125\\
145&175&214&58&76&97\\
170&220&107&229&19&20\\
53&54&248&9&199&202\\
66&89&31&240&173&166\\
76&99&36&215&184&155
\end{array}\right]
\\[17mm]\hspace*{-1.1mm}
\left[
\begin{array}{@{\,}c@{\;\,}c@{\;\,}c@{\;\,}c@{\;\,}c@{\;\,}c@{\,}}
1023&513&526&51&452&504\\
584&698&855&230&305&397\\
677&884&427&920&79&82\\
219&213&992&44&794&807\\
256&363&124&963&701&662\\
310&398&145&861&738&617
\end{array}\right]
&
\left[
\begin{array}{@{\,}c@{\;\,}c@{\;\,}c@{\;\,}c@{\;\,}c@{\;\,}c@{\,}}
2047&1025&1038&119&920&992\\
1200&1378&1691&462&597&813\\
1372&1773&850&1841&170&135\\
451&413&1956&56&1638&1627\\
516&762&245&1923&1353&1342\\
555&790&361&1740&1463&1232
\end{array}\right]
&
\left[
\begin{array}{@{\,}c@{\;\,}c@{\;\,}c@{\;\,}c@{\;\,}c@{\;\,}c@{\,}}
4095&2055&2105&194&1864&1972\\
2448&2788&3407&798&1137&1707\\
2794&3384&1747&3909&421&30\\
769&891&3732&253&3210&3430\\
1063&1490&492&3624&2975&2641\\
1116&1677&802&3507&2678&2505
\end{array}\right]
\end{array}\hspace*{-1.1mm}
\]
\[
\left[
\begin{array}{cccccc}
16383&8207&8433&1814&6432&7880\\
8746&11717&13136&3512&5743&6295\\
11284&14248&7203&15067&965&382\\
2788&4754&15694&1097&10683&14133\\
4441&3955&2958&14885&13468&9442\\
5507&6268&1725&12774&11858&11017
\end{array}\right]
\qquad
\left[
\begin{array}{cccccc}
32767&16415&16865&3622&13896&14736\\
22147&19288&31844&4602&11141&9279\\
22568&26566&12819&32157&2294&1897\\
11123&7461&25276&4813&21778&27850\\
1236&15018&3467&26416&29039&23125\\
8460&13553&8030&26691&20153&21414
\end{array}\right]
\]

For orders larger than 6 it is not feasible to do exhaustive
computations.  However, we did a partial enumeration of MOFS$(10)$
inspired by the example in~\eref{e:9maxMOFSnorel}. The idea was to
impose a block circulant structure similar to that example. Each
square was assumed to be composed of 4 circulant blocks. Under
this (strong) assumption, we found that the largest set of MOFS$(10)$
that is possible has size 17. Every such example 
satisfies a $(5,5)$-relation, and hence is maximal by~\tref{t:relmax}.
The first and sixth rows of one such example are
\[
\left[
\begin{array}{cccccccccc}
52452&86882&89113&107209&108822&26453&27322&38362&39725&79015\\
131071&127&3971&29068&46640&63555&72404&77160&115121&116238
\end{array}
\right].
\]
In light of \tref{t:restrictrel}, the only other $k$ for which we
might hope to find a block circulant $k$-MOFS$(10)$ satisfying
a $(5,5)$-relation are $k=1$ and $k=9$.  The former case is rather
trivially handled by \tref{t:bach}, whilst for $k=9$ we did find a
(necessarily maximal) set with the following first and sixth rows
\[
\left[
\begin{array}{cccccccccc}
210&332&353&404&427&110&117&157&162&283\\
511&1&14&55&248&201&312&338&420&455
\end{array}
\right].
\]

\section{Embeddings}\label{s:embed}

As discussed in the introduction, the Hadamard conjecture implies the
existence of complete MOFS of type $F(n;n/2)$ whenever $n$ is
divisible by $4$. In this section we explore the case $n\equiv2\pmod4$ 
via embeddings of MOFS, building on the computational results
in the previous section.  The following lower bounds for the number of
binary MOFS of order $n\equiv 2\pmod4$ for small values of $n$ are
given in \cite{LM,LZD}:

\begin{theorem}\label{t:prevbnd}
There exist $k$-{\rm MOFS}$(n)$ whenever $(k,n)$ is an element of
\begin{align*}
\{&(8,6),(4,10),(4,14),(8,18),(4,22),(4,26),(8,30),(4,34),(4,38),(8,42), (5,46),(6,50), \\
&(7,54),(5,58),(6,62),(7,66),(6,70),(7,74),(7,78),(8,82),(6,86),(8,90),(7,94),(6,98)\}.
\end{align*}
\end{theorem}
\noindent
Theorem~\ref{newbounds} at the end of this section improves each of the lower bounds in \tref{t:prevbnd} to $k=17$. 

Let $s,n$ be positive even integers with $s<n$. We define an
\emph{incomplete frequency square} of type $(n;s)$ to be an
$n\times n$ array $F$, indexed by $N(n)$, such that:
\begin{enumerate}
\item the subarray indexed by $N(s)\times N(s)$ is empty and all
other cells of $F$ contain $0$ or $1$, and
\item each row and column is \emph{balanced} in the sense that
it contains equal numbers of the symbols $0$ and $1$.
\end{enumerate}

Two such incomplete frequency squares are said to be \emph{orthogonal}
if, when superimposed, each of the four possible ordered pairs
$(0,0)$, $(0,1)$, $(1,0)$ and $(1,1)$ occurs the same number of times.
We use the notation $k$-IMOFS$(n;s)$ to denote a set of $k$ incomplete
frequency squares, each pair of which is orthogonal in the above
sense. We will present IMOFS in superimposed format, similar to
\eref{e:5maxMOFSrel}.  Note that similar results to ours below could
be developed for IMOFS with multiple holes; however a single hole is
enough for the purposes of this paper.

In the following, for any binary vector ${\bf r}$, we write
${\bf\overline{r}}$ for the complement of ${\bf r}$, that is, the vector
formed by replacing each entry $e$ with $1-e$. The array $I({\bf r})$
is the $2\times 2$ array defined by
$$
I({\bf r})=
\left[
\begin{array}{cc}
{\bf r} & \overline{\bf r} \\
\overline{\bf r} & {\bf r} \\
\end{array}
\right].
$$
An important property of $I({\bf r})$ is that its rows and columns are balanced.

\begin{lemma}\label{yeahyeah}
If there exists a $2$-{\rm IMOFS}$(n;n-2)$, then the bottom
right-hand corner must be isomorphic to
$$
\left[
\begin{array}{cc}
00 & 01 \\
10 & 11 \\
\end{array}
\right].
$$
\end{lemma}

\begin{proof}
The number of filled cells in a pair of IMOFS$(n;n-2)$ is equal to $4n-4$.
Therefore there are an odd number of cells filled with ${\bf  r}$
for each of the four choices of ${\bf r}$.  
Observe that to achieve balance in any row or column, 
the number of occurrences of ${\bf r}$ must equal the
number of occurrences of ${\bf \overline{r}}$ for each possible ${\bf r}$.
Let $M$ be the $2\times 2$ subarray in the bottom right-hand corner.
Suppose, for the sake of contradiction, that there exists an ordered
pair ${\bf r}$ which does not occur in $M$.  Then, without loss of
generality (considering the transpose if necessary), ${\bf r}$ occurs
$\alpha$ times within the first $n-2$ rows where $\alpha$ is odd.
Now, each of the first $n-2$ rows contains ${\bf r}$ if and only if it
contains ${\bf \overline{r}}$.  Thus, ${\bf \overline{r}}$ occurs
$\alpha$ times within the first $n-2$ rows.  Since ${\bf r}$ and
${\bf\overline{r}}$ must occur the same number of times in the final
two columns, ${\bf \overline{r}}$ cannot occur in $M$.  However,
considering the final column, there must be an equal number of
occurrences of ${\bf r}$ and ${\bf \overline{r}}$, contradicting the
fact that $\alpha$ is odd.  Thus each possibility for ${\bf r}$ occurs
in $M$ exactly once.  Next, suppose that ${\bf r}$ and
${\bf\overline{r}}$ occur in the final column of $M$.  Then ${\bf r}$
occurs an even number of times in the first $n-2$ rows.  However,
${\bf r}$ must occur an odd number of times in the first $n-2$
columns. Thus ${\bf r}$ occurs an even number of times altogether, a
contradiction. A similar argument shows that ${\bf r}$ and
${\bf\overline{r}}$ cannot occur in any row or column of $M$, from
which the result follows.
\end{proof}

\begin{lemma} 
There does not exist a $3$-{\rm IMOFS}$(n;n-2)$.
\end{lemma}

\begin{proof}
From \lref{yeahyeah}, without loss of generality the bottom right-hand
corner must be isomorphic to:
$$
\left[
\begin{array}{cc}
000 & 01a \\
10b & 11c \\
\end{array}
\right]
$$ 
Considering the first and third IMOFS in light of \lref{yeahyeah},
we must have $a=c=1$. However, considering the second and
third IMOFS, we get $a=0$, a contradiction.
\end{proof}

We now turn our attention to existence results. 

\begin{lemma}
There exists a pair of {\rm IMOFS}$(n;n-2)$ for each even $n\geq 4$.
\end{lemma}

\begin{proof}
We first exhibit 2-{\rm IMOFS}$(4,2)$:
$$
\left[
\begin{array}{cccc}
  \cdot & \cdot & 01 & 10 \\ 
  \cdot & \cdot & 11 & 00 \\
  10 & 11 & 00 & 01 \\
  01 & 00 & 10 & 11 \\
\end{array}
\right].
$$
For $n\geq 6$, the above structure can be placed in the
bottom right-hand corner. Fill the remaining cells in the last two rows
with copies of $I(00)$ and the remaining cells in the last two columns
with copies of $I(01)$.
\end{proof}

A $(v_*,k_*,\lambda_*)$ orthogonal array is a $\lambda_*v_*^2 \times k_*$ 
array with entries chosen from a set $X$ of size $v_*$ such that
in every pair of columns of the array, each ordered pair from $X$
occurs exactly $\lambda_*$ times.  Let $H$ be a Hadamard matrix of
order $4\lambda_*$ in normalized form (so that the first column are
all $1$'s).  Then by the definition of a Hadamard matrix, the
remaining columns of $H$ form a $(2,4\lambda_*-1,\lambda_*)$
orthogonal array (with $X=\{1,-1\}$).  Existence results for Hadamard
matrices (see \cite{CK}) yield the following:

\begin{lemma}\label{eggshists}
There exists a $(2,4\lambda_*-1,\lambda_*)$ orthogonal array whenever 
$1\leq \lambda_*<167$ or $\lambda_*$ is a power of $2$.
\end{lemma}

\begin{theorem}\label{mbed}
If there exists a $(2,k_*,\lambda_*)$ orthogonal array and $4\lambda_*$
divides $b(n-b)$, then there exists a $k_*$-{\rm IMOFS}$(n;n-2b)$.
\end{theorem}

\begin{proof}
There are $4b(n-b)$ non-empty cells in an {\rm IMOFS}$(n;n-2b)$.  Let
$\alpha=4b(n-b)/16\lambda_*$.  Fill the non-empty cells using $\alpha$
copies of $I({\bf r})$ for each row ${\bf r}$ of the
$(2,k_*,\lambda_*)$ orthogonal array (we are assuming without loss of
generality that the symbols of the orthogonal array are $0$ and $1$).
\end{proof}

By \lref{eggshists} and \tref{mbed}, we get:

\begin{corollary}\label{boost}
If $2^{\beta}$ divides $b(n-b)$, then there exists 
a $(2^{\beta}-1)$-{\rm IMOFS}$(n;n-2b)$. 
\end{corollary}

It is worth noting that if $b$ is odd, the previous theorem is of
little use.  That is, the embedding approach in this section is not
apparently helpful in obtaining set of MOFS of order not divisible by
$4$ from sets of MOFS of order divisible by $4$, the latter of which
are far easier to construct.

\begin{theorem}\label{newbounds}
There exists a $17$-{\rm MOFS}$(n)$ for each 
order $n\equiv 2\pmod{4}$ where $n\geq 6$. 
\end{theorem} 

\begin{proof}
From the previous section, there exists a $17$-MOFS$(6)$ and a
$17$-MOFS$(10)$.  Now, from \lref{eggshists}, there exist a
$(2,k_*,\lambda_*)$ orthogonal array with $k_*\geq 17$ for
$\lambda_*\in \{6,14,24,36,50,66\}$.  By \tref{mbed}, there thus
exists a $17$-IMOFS$(10+4B;10)$ for $1\leq B\leq 6$.
Thus, by ``plugging'' the hole of size $10$ with a $17$-MOFS$(10)$,
there exists a $17$-MOFS$(n)$ for each $n\equiv 2\pmod{4}$ such that
$14\leq n\leq 34$.

Next, if $16$ divides $b$, then by Corollary~\ref{boost}, there exists
a $31$-IMOFS $(n;n-2b)$ for any $n>2b$. The result follows
recursively.
\end{proof}

\section{Maximal sets of MOFS}\label{s:max}

In this section we show the following existence result.

\begin{theorem}\label{t:5max}
There exists a $5$-{\rm maxMOFS}$(n)$ for
each order $n\equiv 2\pmod{4}$ where $n\geq 6$. 
\end{theorem}

\begin{proof}
Our starting point is the $5$-maxMOFS$(6)$ satisfying a
$(5,3)$-relation that was given in \eref{e:5maxMOFSrel}.  We now
embed that $5$-maxMOFS$(6)$ into a $5$-maxMOFS$(4\kappa+2)$, for each
$\kappa>1$. In the process we will add each binary $5$-tuple the same
number of times, thereby ensuring that orthogonality is preserved.
Also the resulting $5$-MOFS$(4\kappa+2)$ will satisfy a
$(2\kappa+3,2\kappa+1)$-relation, ensuring that it is maximal, by
\tref{t:relmax}. 
Let $X_1$ be the set of the first $2\kappa+3$ rows and $X_2$ the set of the
first $2\kappa+1$ columns.  Let $X_1'$ and $X_2'$ be the sets of rows and 
columns not in $X_1$ and $X_2$, respectively.  Our MOFS will have 
$\Z_2$-sum given by \eref{e:blocks}, where the top left block has rows $X_1$
and columns $X_2$. 

Next we describe the placement of the $5$-maxMOFS$(6)$. These are
placed in the first three columns of each of $X_2$ and $X_2'$, the
first 5 rows of $X_1$ and the first row of $X_1'$.  Let $C$ be the set
of cells which do not include the 36 cells just specified. Then
$|C|=(4\kappa+2)^2-6^2=16(\kappa^2+\kappa-2)$.  Excluding the first
four rows of $C$, observe that the remaining cells may be partitioned
into intercalates in \eref{e:blocks}, with one row in each of $X_1$
and $X_1'$ and one column in each of $X_2$ and $X_2'$.

We complete our construction by describing how to fill the remaining
cells in each frequency square.  We do so by first describing how to
fill the first four rows then the remaining cells of $C$ using a
partition of intercalates as described above.

First suppose that $\kappa\equiv 1$ or $2\pmod{4}$. Then $64$ divides
$|C|$; let $c=|C|/64$.  Observe that
$c=(\kappa^2+\kappa-2)/4\geq\kappa-1$.  We fill the cells of $C$ in
the first $4$ rows with $\kappa-1$ copies of the following array 
so that the tuples with an even number of ones occur in columns in
$X_2$ and tuples with an odd number of ones occur in columns in $X_2'$:
\begin{equation}\label{e:Aarray}
\left[
\begin{array}{cccc}
00011&00011 &11100&11100 \\
01100&01100 &10011&10011 \\
10001&10001 &01110&01110 \\
11110&11110 &00001&00001 \\
\end{array}
\right]
\end{equation}
Note that this array has balanced rows and columns.
There are 24 binary $5$-tuples not in \eref{e:Aarray}, which in turn
partition into 12 complementary pairs.  
Let ${\bf r_i}$ for $1\leq i\leq 12$ be the representatives from these pairs
which contain an even number of ones.

To fill the intercalates we first add $\kappa-1$ copies of each
$I({\bf r_i})$ to $C$.  We have thus far filled $64(\kappa-1)$
cells of $C$, including the cells in the first $4$ rows, with each
binary $5$-tuple occurring exactly $2(\kappa-1)$ times.  To fill the
remaining $64(c-\kappa+1)$ cells of $C$, we partition \emph{all} 32
binary $5$-tuples into 16 complementary pairs, represented by ${\bf r_i}$ 
for $1\leq i\leq 16$, and add $c-\kappa+1 \geq 0$ copies of
$I({\bf r_i})$ for each of the 16 possible values of $i$.  It is now
routine to check that the construction indeed results in
$5$-MOFS$(4\kappa+2)$ whose $\Z_2$-sum is \eref{e:blocks}.

Otherwise $\kappa\equiv 3$ or $0\pmod{4}$.  In this case
$16(\kappa^2+\kappa-2)\equiv 32\pmod{64}$.  Consider the following
array~$B$:
$$
\left[
\begin{array}{cccc}
11100&11010&11111&11001\\
10011&10101&10000&10110\\
01110&01011&00111&01101\\
00001&00100&01000&00010\\
\end{array}
\right]
$$
This array has balanced columns and 
includes every binary $5$-tuple with an odd number
of~$1$'s.  Therefore, if we take the complement of each tuple we
obtain an array $\overline{B}$ including every $5$-tuple with an even
number of~$1$'s.  Place exactly one copy of $B$ and ${\overline B}$ in
the first $4$ rows of $C$, with ${\overline B}$ in columns of $X_2$
and $B$ in columns of $X_2'$.  Note that this is possible since $\kappa\geq
3$.  The number of remaining cells is divisible by $64$ so we can
proceed as in the previous case.
\end{proof}

\section{Concluding remarks}\label{s:conc}

\tref{t:bach} showed that for $n\equiv2\pmod4$ there is a unique
bachelor frequency square. Our computations showed that there is no
other maximal $k$-MOFS$(6)$ with $k<5$. It would be very
interesting to know whether this generalises to larger
$n\equiv2\pmod4$.  Note that we do know that there is a maximal 
$5$-MOFS$(n)$, by \tref{t:5max}.  For $n\equiv0\pmod4$, the question
of how small a maximal set of MOFS$(n)$ can be, is wide open.  If it
turns out that there are no maximal sets with fewer than 5 MOFS aside
from those in \tref{t:bach} then that would be a significant difference
from Latin squares.  It is known \cite{DWW11} that maximal pairs of
mutually orthogonal Latin squares exist for all orders $n>6$ that are
not twice a prime.

In \tref{t:manyMOFS} we gave a lower bound on the number of complete
MOFS$(n)$ for $n\equiv0\pmod4$ (assuming the Hadamard conjecture). It would
be interesting to obtain a corresponding upper bound. In particular, it would
be nice to know whether the exponent in our bound is of the correct order.
Note that the corresponding problem for Latin squares has very recently
been solved \cite{BDS}.

Two interesting directions for possible generalisation of our results
are to frequency squares with 2 symbols that do not occur equally
often or to frequency squares with more than 2 symbols. In particular,
how many symbols does it take before bachelor frequency squares become
common and other small maximal sets become possible?

\subsection*{Acknowledgments}
This work was supported in part by Australian Research Council grant
DP150100506.  The first author thanks Daniel Mansfield for invaluable
C programming help.

  \let\oldthebibliography=\thebibliography
  \let\endoldthebibliography=\endthebibliography
  \renewenvironment{thebibliography}[1]{%
    \begin{oldthebibliography}{#1}%
      \setlength{\parskip}{0.4ex plus 0.1ex minus 0.1ex}%
      \setlength{\itemsep}{0.4ex plus 0.1ex minus 0.1ex}%
  }%
  {%
    \end{oldthebibliography}%
  }

\end{document}